\documentclass[11pt]{amsart}

\usepackage{xcolor}

\usepackage[margin=1in]{geometry}

\usepackage{amsmath}
\usepackage{amssymb}
\usepackage{amsfonts}
\usepackage{amsthm}
\usepackage{mathtools}
\usepackage{bm}
\usepackage{mathrsfs}
\usepackage{esint}

\allowdisplaybreaks[1]

\usepackage{enumitem}

\usepackage{aliascnt}

\theoremstyle{plain}
\newtheorem{theorem}{Theorem}[section]

\newaliascnt{lemma}{theorem}
\newtheorem{lemma}[lemma]{Lemma}
\aliascntresetthe{lemma}

\newaliascnt{proposition}{theorem}
\newtheorem{proposition}[proposition]{Proposition}
\aliascntresetthe{proposition}

\newaliascnt{corollary}{theorem}
\newtheorem{corollary}[corollary]{Corollary}
\aliascntresetthe{corollary}

\newaliascnt{conjecture}{theorem}

\aliascntresetthe{conjecture}

\theoremstyle{definition}

\newaliascnt{definition}{theorem}
\newtheorem{definition}[definition]{Definition}
\aliascntresetthe{definition}

\newaliascnt{assumption}{theorem}

\aliascntresetthe{assumption}

\newaliascnt{example}{theorem}

\aliascntresetthe{example}

\newaliascnt{notation}{theorem}

\aliascntresetthe{notation}

\theoremstyle{remark}

\newaliascnt{remark}{theorem}
\newtheorem{remark}[remark]{Remark}
\aliascntresetthe{remark}

\numberwithin{equation}{section}

\DeclarePairedDelimiterX{\inner}[2]{\langle}{\rangle}{#1,#2}

\newcommand{\dd}{\mathop{}\!\mathrm{d}}

\usepackage[
    pagebackref=true,
    colorlinks=true,
    linkcolor=blue,
    citecolor=blue,
    urlcolor=blue
]{hyperref}

\usepackage[nameinlink]{cleveref}

\crefname{theorem}{theorem}{theorems}
\Crefname{theorem}{Theorem}{Theorems}
\crefname{lemma}{lemma}{lemmas}
\Crefname{lemma}{Lemma}{Lemmas}
\crefname{proposition}{proposition}{propositions}
\Crefname{proposition}{Proposition}{Propositions}
\crefname{corollary}{corollary}{corollaries}
\Crefname{corollary}{Corollary}{Corollaries}
\crefname{conjecture}{conjecture}{conjectures}
\Crefname{conjecture}{Conjecture}{Conjectures}
\crefname{definition}{definition}{definitions}
\Crefname{definition}{Definition}{Definitions}
\crefname{assumption}{assumption}{assumptions}
\Crefname{assumption}{Assumption}{Assumptions}
\crefname{example}{example}{examples}
\Crefname{example}{Example}{Examples}
\crefname{notation}{notation}{notations}
\Crefname{notation}{Notation}{Notations}
\crefname{remark}{remark}{remarks}
\Crefname{remark}{Remark}{Remarks}

\usepackage{mathrsfs}
\title
{The quermassintegral inequalities
for \( k \)-convex {\em spiral-star-shaped} domains}

\author{Junfang Li}
\address{Department of Mathematics\\
University of Alabama at Birmingham\\
Birmingham, AL 35294}
\email{jfli@uab.edu}

\subjclass[2020]{53E10, 53C42, 35K55}
\keywords{Inverse curvature flow, 
  Alexandrov--Fenchel quermassintegral inequality, 
{\em spiral-star-shaped} domains}

\date{}

\begin{document}
\begin{abstract}
We extend the quermassintegral inequalities of our earlier work~\cite{GuanLi2009}
from star-shaped domains to a substantially broader class of domains in
Euclidean space, which we call {\em spiral-star-shaped}. A smooth domain is
spiral-star-shaped if its boundary satisfies
\[
  \langle X+\mathsf A X,\nu\rangle>0
\]
for some fixed skew-symmetric linear map \(\mathsf A\), where \(X\) is the
position vector and \(\nu\) is the outward unit normal. The vector field
\(X+\mathsf A X\) is conformal Killing, with dilation component \(X\) and
rotational component \(\mathsf A X\); its trajectories combine expansion and
rotation, giving rise to the spiral terminology.  
When \(\mathsf A=0\), it reduces to ordinary
star-shapedness, while nonzero \(\mathsf A\) allows domains that are not
star-shaped in the usual radial sense.

For smooth \(k\)-convex spiral-star-shaped domains, we prove long-time
existence and exponential convergence to a round sphere for a locally
constrained inverse curvature flow. The positive spiral support function
provides the geometric control needed to treat initial hypersurfaces that
need not be star-shaped in the usual sense. Monotonicity of the
quermassintegrals along the flow yields sharp quermassintegral inequalities
in this enlarged class, with equality if and only if the domain is a ball.
\end{abstract}

\maketitle

\section{Introduction}
\label{sec:introduction}

The Alexandrov--Fenchel inequalities for the quermassintegrals of
convex domains are fundamental in classical geometry. Extending
these inequalities beyond convexity has motivated the development
of methods from fully nonlinear partial differential equations,
geometric flows, optimal transport, and potential theory. In
Euclidean space, Guan and Li \cite{GuanLi2009} established the sharp
quermassintegral inequalities for $k$-convex star-shaped domains.
Their proof uses the inverse curvature flows studied by Gerhardt
and Urbas \cite{Gerhardt1990,Urbas1990}, whose long-time existence
and asymptotic convergence convert monotonicity of curvature
integrals into sharp geometric inequalities. In particular, for
a smooth \( k \)-convex star-shaped domain
$\Omega\subset\mathbb R^{n+1}$, the inverse curvature flow
yields the celebrated Alexandrov-Fenchel quermassintegral inequalities.
For $0\le m\le n$,
\begin{equation}\label{AFI}
\Big(\frac{V_{(n+1)-m}(\Omega)}{V_{(n+1)-m}(B)}\Big)^{\frac{1}{n+1-m}}\le
\Big(\frac{V_{n-m}(\Omega)}{V_{n-m}(B)}\Big)^{\frac{1}{n-m}},
\end{equation} where $B$ is the standard ball in $\mathbb R^{n+1}$.
The equality holds if and only if $\Omega$ is a ball. The case $m=0$
is the classical isoperimetric inequality and \( m=1 \) is the
classical Minkowski inequality. 
Huisken~\cite{Huisken2007} proved that Minkowski inequality holds
for outward-minimizing hypersurfaces.

Since the work~\cite{GuanLi2009}, curvature-flow methods have been developed
extensively in other ambient geometries and through alternative
choices of flow. There have also been important advances in the
original Euclidean setting. Chang and Wang
\cite{ChangWang2013,ChangWang2014}, and subsequently Qiu
\cite{Qiu2015}, obtained inequalities for nonconvex domains using
optimal transport. Their results include estimates for
quermassintegrals and sharp inequalities comparing volume with
higher mean-curvature integrals under suitable curvature
assumptions. Weak inverse mean curvature flow and nonlinear
potential theory provide a different approach to the Minkowski
inequality for outward-minimizing domains. In particular,
Agostiniani, Fogagnolo, and Mazzieri
\cite{AgostinianiFogagnoloMazzieri2022} established an extended
Minkowski inequality involving the outward-minimizing hull,
recovering the classical inequality for smooth outward-minimizing
domains. These developments broaden the range of domains for
which geometric inequalities are available, while leaving a
distinct question concerning smooth evolution: which geometric
conditions permit inverse curvature flow to exist for all
time and converge, after rescaling, to a round sphere?

Ordinary star-shapedness plays a central role in the classical
smooth theory by providing a global radial parametrization and
control of the support function. Nevertheless, it is not necessary
for smooth global evolution. Harvie \cite{Harvie2022, Harvie2023} proved
long-time existence and rescaled convergence for a class of
non-star-shaped, mean-convex hypersurfaces with axial symmetry and
suitable geometric conditions on their necks, obtaining the
Minkowski inequality for the corresponding domains. It is therefore
natural to seek further conditions that allow non-star-shaped
initial hypersurfaces without imposing symmetry.

The present work is motivated by Flynn and Reznikov
\cite{FlynnReznikov2023}, who introduced a modification of the
Guan--Li mean curvature type flow \cite{GuanLi2015} adapted to
conformal Killing fields perturbed by other Killing fields.
Our starting point is the following locally
constrained inverse-curvature flow in Euclidean space.
The space form versions of this flow can be found in our 
work~\cite{GuanLi2026} or its earlier version~\cite{BGL}.

\begin{equation}
    \partial_t X
    =
    \left(
        \frac{1}{F}
        -\frac{u}{c_{n,k}}
    \right)\nu,
    \qquad
    F=\frac{\sigma_{k+1}}{\sigma_k},
    \label{eflow}
\end{equation}
where
\[
    u=\langle X,\nu\rangle,
    \qquad
    c_{n,k}=\frac{n-k}{k+1}.
\]

We prove that smooth strictly \( k \)-convex {\em spiral-star-shaped } 
initial hypersurfaces admit a smooth inverse curvature flow~\eqref{eflow}
for all time, and that the evolving hypersurfaces converge
exponentially to a round sphere. 
The spiral support function supplies the geometric control needed to extend
the classical smooth theory to this class. 
As a consequence, we obtain the sharp quermassintegral inequalities for
{\em spiral-star-shaped} strictly \( k \)-convex domains, 
extending the corresponding star-shaped case of \cite{GuanLi2009}.
The result identifies a first-order geometric condition under which the
Euclidean inverse curvature flow itself provides a smooth
deformation to a sphere and a direct proof of the inequality.

A principal difference from the approach of
Flynn--Reznikov~\cite{FlynnReznikov2023} is that we do not modify the
flow by inserting the rotational Killing field into our evolution equation. 
Instead, we keep exactly the inverse-curvature flow
\eqref{eflow} from~\cite{GuanLi2026,GuanLi2009} and enlarge the class of initial
hypersurfaces on which the flow can be analyzed.

More precisely, let
\[
    \mathsf A\in\mathfrak{so}(n+1)
\]
be fixed and let
\[
    \mathcal K(X)=\mathsf A X
\]
be the corresponding rotational Killing field. We replace the usual
radial support function by
\[
    \widehat u
    :=
    \langle X+\mathcal K(X),\nu\rangle.
\]
The condition
\[
    \widehat u>0
\]
means that the hypersurface is transverse to the integral curves of
the conformal Killing field
\[
    Y(X):=X+\mathcal K(X).
\]
These integral curves are obtained by combining dilation and rotation,
and in dimension two they are logarithmic spirals. We therefore refer
to such hypersurfaces as {\em spiral-star-shaped}.

This condition expresses transversality
to the expanding spiral trajectories of $Y$. It includes ordinary
star-shapedness as the case $\mathsf A=0$, while allowing
hypersurfaces that are not radial graphs.

\begin{figure}[!htb]
   \begin{minipage}{0.48\textwidth}
     \centering
     \includegraphics[width=.7\linewidth]{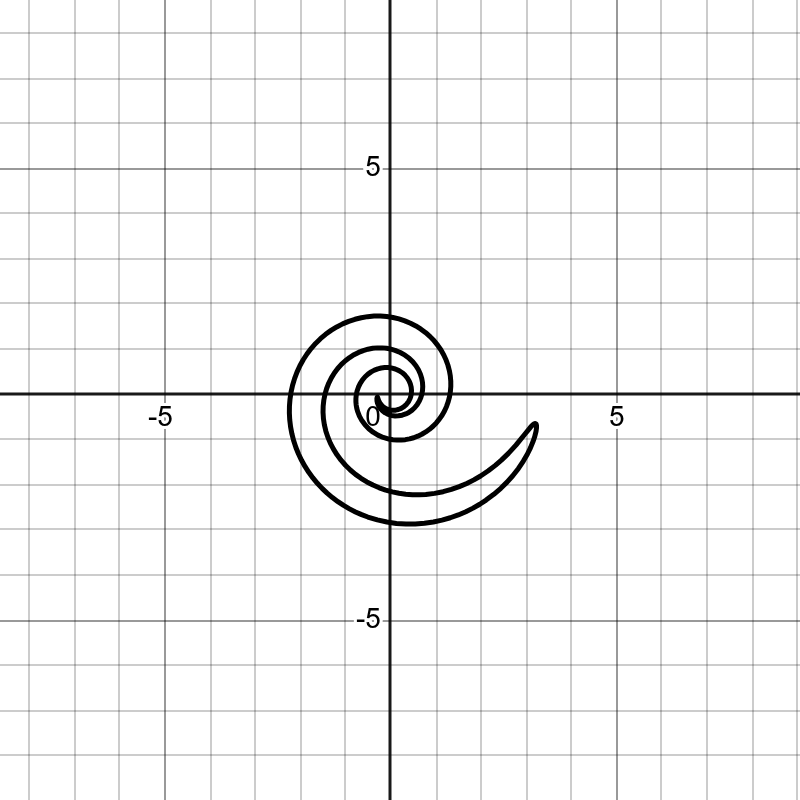}
     \caption{}
     \label{spiral1}
   \end{minipage}\hfill
\end{figure}

\medskip

In this paper, we prove the following main theorem.

\begin{theorem}
\label{thm:main}
Let \(0\le k\le n-1\), and let
\(M_0\subset\mathbb R^{n+1}\) be a smooth closed embedded hypersurface
whose principal-curvature vector satisfies
\[
    \kappa[M_0]\in\Gamma_{k+1}.
\]
Suppose furthermore that \(M_0\) is {\em spiral-star-shaped} with
respect to some fixed
\[
    \mathsf A\in\mathfrak{so}(n+1),
\]
namely,
\[
    \widehat u
    =
    \langle X+\mathsf A X,\nu\rangle
    >0
    \qquad\text{on }M_0.
\]
Then the flow~\eqref{eflow} admits a unique smooth admissible solution
for all
\[
    t\in[0,\infty).
\]
Moreover, \(M_t\) converges exponentially in \(C^\infty\) to a round
sphere centered at the origin as \(t\to\infty\).
\end{theorem}

\begin{remark}
When
\[
    u=\langle X,\nu\rangle>0,
\]
the hypersurface is transverse to the radial rays
\[
    s\longmapsto e^s\theta.
\]
By contrast,
\[
    \widehat u
    =
    \langle X+\mathsf A X,\nu\rangle>0
\]
means that \(M\) is transverse to the integral curves
\[
    s\longmapsto e^s e^{s\mathsf A}\theta
\]
of the conformal Killing field \(Y=X+\mathsf A X\). Thus rotational
star-shapedness replaces the radial foliation by a simultaneous
dilation--rotation foliation.
\end{remark}

The main observation is that the rotational geometry can be absorbed
entirely into the parametrization of the hypersurface. Writing a
{\em spiral-star-shaped} hypersurface as a graph along the integral
curves of \(Y\), we derive a scalar parabolic equation whose geometric
evolution is still exactly \eqref{eflow}. This allows the
a priori estimates and convergence theory developed in~\cite{GuanLi2026, GuanLi2021, BGL} to
be combined with new estimates for the rotational support function.

In particular, we establish uniform \(C^0\) and gradient estimates,
two-sided bounds for \(F\), and uniform curvature estimates. These
estimates imply long-time existence and exponential convergence of
the flow to a round sphere. Thus, at least in Euclidean space, the
inverse-curvature flow of~\cite{GuanLi2026} naturally extends from radial
star-shaped hypersurfaces to the much larger {\em spiral-star-shaped}
class.

As applications, we generalize the main theorem of~\cite{GuanLi2009}
to {\em spiral-star-shaped} domains in \( \mathbb R^{n+1} \).

\begin{theorem}
\label{thm:afi-inequalities}
Suppose \( \Omega \) is a smooth \( k \)-convex spiral-star-shaped domain
in \(\mathbb R^{n+1}  \) with respect to
\[
    \mathsf A\in\mathfrak{so}(n+1),
\]
then the classical quermassintegral inequalities~\eqref{AFI}
of~\cite{GuanLi2009,Trudinger1994} hold for all \( 0\le m\le k \).
The equality holds if and only if \( \Omega \) is a ball.
\end{theorem}

The rest of the paper is organized as follows.
In~\Cref{sec:spiral-star-shaped-domains} we introduce 
{\em spiral-star-shaped} domains and discuss the basic properties. 
In~\Cref{sec:spiral-radial-parametrization}, 
we introduce and construct the associated spiral graph parametrization.
In~\Cref{sec:ivp-C0}, we reduce the flow to be a scalar initial-value problem
and obtain its \(C^0\) estimate.
In~\Cref{sec:gradient-estimate}, we proved \( C^1 \) estimate which also leads to
a uniform positive lower bound for \(F\).
In~\Cref{sec:curvature-estimate} we combine these estimates with
the curvature evolution equations from~\cite{GuanLi2026} to obtain uniform
\(C^2\) estimates.
Standard parabolic regularity and the convergence
argument of~\cite{GuanLi2026, GuanLi2021, BGL} then yield long-time existence and exponential
convergence.

\section{Properties of spiral-star-shaped domains}
\label{sec:spiral-star-shaped-domains}

Let
\[
    \mathsf A\in\mathfrak{so}(n+1)
\]
be a fixed skew-symmetric linear map. The associated rotational
Killing field on \(\mathbb R^{n+1}\) is
\[
    \mathcal K(x):=\mathsf A x.
\]
We combine this field with the Euclidean dilation field \(X\) and set
\[
    Y(X):=X+\mathcal K(X)=X+\mathsf A X.
\]
Clearly, the vector field \(Y\) is conformal Killing.

\begin{definition}[Spiral support function]
\label{def:spiral-support}
Let \(M\subset\mathbb R^{n+1}\) be a smooth oriented hypersurface with
position vector \(X\) and outward unit normal \(\nu\). We define the
spiral support function associated with \(\mathsf A\) by
\begin{equation}
    \widehat u
    :=
    \langle Y(X),\nu\rangle
    =
    \langle X+\mathsf A X,\nu\rangle.
    \label{eq:spiral-support}
\end{equation}
Equivalently,
\[
    \widehat u=u+\mathring u,
\]
where
\[
    u:=\langle X,\nu\rangle,
    \qquad
    \mathring u:=\langle\mathsf A X,\nu\rangle
\]
are respectively the radial and rotational parts of the support
function.
\end{definition}

\begin{definition}[Spiral-star-shapedness]
\label{def:spiral-star-shaped}
A smooth closed hypersurface \(M\subset\mathbb R^{n+1}\) is said to be
spiral-star-shaped with respect to \(\mathsf A\) if
\[
    \widehat u>0
    \qquad\text{on }M.
\]
When \(\mathsf A=0\), this reduces to the usual radial star-shaped
condition
\[
    u=\langle X,\nu\rangle>0.
\]
\end{definition}

The first property is an application of the fact that 
when \( \mathsf A \) is skew-symmetric, the rotational Killing 
field \( \mathcal K(x)=\mathsf Ax \) is divergence-free.

\begin{lemma}
    \label{lem:K-divergence-free} 

Let
\(
    \mathsf A\in\mathfrak{so}(n+1)
\)
be a fixed skew-symmetric linear map.
Suppose \( \Omega \subset \mathbb R^{n+1} \) is a bounded domain. 
Then
\begin{align*}
    \int_{\partial \Omega} \mathring u d\mu_g
    =\int_{\partial \Omega} \langle \mathsf A X, \nu\rangle d\mu_g= 0,
\end{align*}
where \( X \) is the position vector and  \( \nu \) is
the outward unit normal.

Moreover,
\begin{align*}
    \int_{\partial \Omega} \widehat u d\mu_g
    =\int_{\partial \Omega} \langle X+ \mathsf A X, \nu\rangle d\mu_g
    =(n+1)|\Omega|.
\end{align*}
\end{lemma}

\begin{proof}
    Let \( \{e_i\}_{i=1}^{n+1} \) be orthonormal basis in \( \mathbb R^{n+1} \). 
    We have
    \begin{align*}
        \mathrm{div}(\mathsf A X) =
        \langle \overline \nabla _{e_i}\mathsf AX, e_i\rangle
        = \langle \mathsf A\overline \nabla _{e_i}X, e_i\rangle
        = \langle \mathsf A e_i, e_i\rangle
        =0,
    \end{align*}
    where we have used that \( \mathsf A \) is skew-symmetric.
    By divergence theorem, we obtain
     \begin{align*}
         \int_{\partial \Omega} \langle \mathsf A X, \nu\rangle d\mu_g
         = \int_{\Omega} \mathrm{div}(\mathsf A X)d \Omega =0.
     \end{align*}
     The second identity follows from the first identity and the classical result
     \begin{align*}
    \int_{\partial \Omega} \langle X, \nu\rangle d\mu_g
    =(n+1)|\Omega|.
     \end{align*}
\end{proof}

\begin{definition}
   Suppose \( \Omega \subset \mathbb R^{n+1} \) is a bounded domain with boundary hypersurface \( M \). We denote 
   \begin{align*}
       \mathcal CM:=
      \bigg \{
           \mathsf A\in \mathfrak{so}(n+1)\ \big|\ 
\widehat u(x)>0 \mathrm{\ with\ respect\ to \ }\mathsf A, \quad
\forall x\in M
       \bigg\}
   \end{align*}
   to be the convex subset of \( \mathsf A \) in \( \mathfrak{so}(n+1) \) 
   with respect to which \( M \) is  {\em spiral-star-shaped}.
\end{definition}

It is immediate that \( \mathcal CM \) is convex. 
Suppose \(\mathsf A, \mathsf B\in \mathcal CM \), then
 \begin{align*}
     \langle X+\mathsf AX,\nu\rangle>0,\quad
     \langle X+\mathsf BX,\nu\rangle>0.
\end{align*}
For any convex combination 
\( t\mathsf A+(1-t)\mathsf B \in \mathfrak{so}(n+1)\),
\( t\in[0,1] \), we have
\begin{align*}
    \langle X+(t\mathsf A+(1-t)\mathsf B)X,\nu\rangle>0.
\end{align*}
Namely, \(  t\mathsf A+(1-t)\mathsf B\in \mathcal CM\).

Another property that compares \( \mathsf A,\mathsf B\in \mathfrak{so}(n+1) \) is the following result.

\begin{proposition}
    \label{prop:order-spiral-support-function}
   For \( \mathsf A,\mathsf B\in \mathfrak{so}(n+1) \), suppose we have
   \begin{align*}
       \widehat u_{\mathsf A}\ge \widehat u_{\mathsf B}.
   \end{align*}
   Then \( (\mathsf A-\mathsf B)X \) is tangential to  \( \partial \Omega \), i.e.
   \begin{align*}
        \langle(\mathsf A-\mathsf B)X, \nu \rangle =0.
   \end{align*}
\end{proposition}
\begin{proof}
    By~\Cref{lem:K-divergence-free}, we have
    \begin{align*}
        \int_{\partial \Omega}\widehat u_{\mathsf A}-\widehat u_{\mathsf B}d\mu_g=0.
    \end{align*}
    By the assumption \( \widehat u_{\mathsf A}-\widehat u_{\mathsf B}\ge 0 \), 
    this implies pointwisely
    \begin{align*}
       \widehat u_{\mathsf A}-\widehat u_{\mathsf B}= \langle(\mathsf A-\mathsf B)X, \nu \rangle =0,
    \end{align*}
    due to continuity of the support functions.
\end{proof}
\section{Spiral radial parametrization}
\label{sec:spiral-radial-parametrization}
In this section, we introduce the spiral radial parametrization.

\subsection{Rotational Killing vector field}
\label{subsec:spiral-Killing-field}

We first clarify our notations.
Let $\mathsf A\in\mathfrak{so}(n+1)$, and define the ambient
rotational Killing field
\[
  \mathcal K(x):=\mathsf A x,
  \qquad x\in\mathbb R^{n+1}.
\]
For $\theta\in\mathbb S^n$, its restriction to the unit sphere is
\[
  \mathcal K(\theta)=\mathsf A\theta
  \in T_\theta\mathbb S^n.
\]
Let $\{e_i\}_{i=1}^n$ be a local orthonormal frame on
$\mathbb S^n$. We use the component notation
\[
  \mathcal K_i
  :=
  \langle\mathcal K(\theta),e_i\rangle_\sigma
  =
  \langle\mathsf A\theta,e_i\rangle.
\]
We also set
\[
  \mathsf A_{ij}
  :=
  \langle\mathsf A e_i,e_j\rangle.
\]
Since $\mathsf A$ is skew-symmetric,
\[
  \mathsf A_{ij}=-\mathsf A_{ji}.
\]

The ambient and intrinsic derivatives of $\mathcal K$ are related by
\[
  \overline\nabla_{e_i}\mathcal K
  =
  \mathsf A e_i,
\]
and
\[
  \nabla_i^\sigma\mathcal K
  =
  \bigl(\mathsf A e_i\bigr)^\top
  =
  \mathsf A e_i+\mathcal K_i\theta,
\]
where we have used skew-symmetry of \( \mathsf A \).
In components, this gives
\[
  \nabla_i^\sigma\mathcal K_j
  =
  \mathsf A_{ij}.
\]

For $\gamma\in C^\infty(\mathbb S^n)$, write
\[
  D\gamma=\gamma^j e_j.
\]
Then
\[
  (\mathsf A D\gamma)_i
  :=
  \langle\mathsf A(D\gamma),e_i\rangle
  =
  \mathsf A_{ji}\gamma^j
  =
  -\mathsf A_{ij}\gamma^j.
\]
Finally, the directional derivative of $\gamma$ along $\mathcal K$ is
\[
  \mathcal K\gamma
  :=
  d\gamma(\mathcal K)
  =
  \langle D\gamma,\mathcal K\rangle_\sigma
  =
  \mathcal K^i\gamma_i.
\]

\subsection{Spiral radial parametrization}
\label{subsec:spiral-radial-parametrization}

Now consider the geodesic normal coordinates of $\mathbb R^{n+1}$
\[
ds^2 = d\rho^2+\rho^2 \sigma,
\]
where $\sigma=d\theta^2 $ is the standard metric of $\mathbb S^n$. We let
 \[
\rho = e^\gamma,\qquad \gamma =\ln \rho.
\]
By direct calculations, we can check that for a fixed $\theta\in \mathbb S^n$
\[
    X=X(s):=e^s\,e^{s\mathsf A }\theta
\]
is an integral curve of the conformal Killing vector field
\[
Y:=X+\mathcal K,
\]
where $e^{s\mathsf A }\, \theta\in \mathbb S^n$ is a one-parameter family of rotations on $\theta$.

Thus, we could parametrize the hypersurface by the following spiral parametrization
\begin{align}
X = \rho \,e^{\gamma \mathsf A}\theta.
    \label{X}
\end{align}
where $\rho=\rho(\theta)$ and $\gamma =\ln\rho(\theta)$ for $\theta\in \mathbb S^n$ if
$M_0$ is spiral-star-shaped.
In particular, for any fixed $\theta$, $X(\cdot,\theta)$
is an integral curve of the conformal Killing field $Y$.

\subsection{Geometric quantities under spiral parametrization}
\label{subsec:geometric-quantities-spiral}
Choose a local orthonormal frame $\{e_i\}$ at a point of the spiral-radial
parametrization \(X = \rho \,e^{\gamma \mathsf A}\theta\). Then
\[
\nabla ^{\sigma}_{e_i}e_j = 0,\quad \overline \nabla _{e_i}e_j=-\sigma_{ij}\theta.
\]
Recall \(e^{\gamma\mathsf A}\) is a linear operator and
\[
    \frac{d}{d \gamma }e^{\gamma\mathsf A}
    = e^{\gamma\mathsf A}\, \mathsf A.
\]

The tangent vector fields
\begin{align*}
    X_i &= \overline \nabla _{e_i}X \\
        &=\gamma _i \rho\, e^{\gamma\mathsf A} \theta
        +\gamma _i\rho\, e^{\gamma\mathsf A} \mathsf A\theta
        +\rho\, e^{\gamma\mathsf A} \overline \nabla _{e_i}\theta\\
        &= \rho\, e^{\gamma\mathsf A}
    \Big(\gamma_i(\mathcal K+\theta)+e_i\Big),
\end{align*}
Thus, the induced metric
\[
g_{ij} = e^{2\gamma} \Big(\sigma_{ij} +\gamma_i\mathcal K_j + \gamma_j\mathcal K_i+\gamma_i \gamma _j( |\mathcal K|^2+1)\Big),
\]
and the inverse metric tensor
\[
g^{ij} = e^{-2\gamma}
\Big(
\sigma^{ij} +\mathcal K^i\mathcal K^j-
\frac{(\gamma ^i+\alpha \mathcal K^i )(\gamma ^j+\alpha \mathcal K^j )}{\omega^2}\Big),
\]
where $\omega := \sqrt{\alpha^2+|D \gamma |^2}$ and \(\alpha:=1+\mathcal K \gamma \).

The unit normal vector field is
\[
    \nu =\frac{1}{\omega}
    e^{\gamma \,\mathsf A}(\alpha\,\theta-D \gamma).
\]
 Consequently,
 \begin{equation}
     \label{eq:U-and-u}
     \widehat u  = \frac{\rho}{\omega},\quad
     \mathrm{and}\quad
u= \alpha\widehat u .
 \end{equation}

Moreover,
\begin{align*}
\overline \nabla _{e_j}X_i
=&\overline \nabla _{e_j}
 \rho\, e^{\gamma\mathsf A}
    \Big(\gamma_i(\mathcal K+\theta)+e_i\Big)\\
    =\rho  e^{\gamma\mathsf A}
    \Big[
 & \gamma _i \gamma_j(\mathcal K+\theta)+\gamma _je_i \\
 &+ \gamma _i \gamma_j \mathsf A(\mathcal K+\theta)+\gamma _j\mathsf A e_i \\
     &+ \gamma_{ij}
    (\mathcal K+\theta)
     + \gamma_{i}
    \overline \nabla _{e_j}(\mathcal K+\theta)
     + \overline\nabla _{e_j}e_i
\Big].
\end{align*}
Thus,
\begin{align*}
\overline \nabla _{e_j}X_i
=\rho  e^{\gamma\mathsf A}
    \Big[
 & \gamma _i \gamma_j(2\mathcal K+\theta)
+ \gamma _i \gamma_j \mathsf A\mathcal K
 \\
 &+\gamma _j(e_i +\mathsf A e_i)
+ \gamma_{i}(e_j+\mathsf A e_j )
 \\
 &
     + \gamma_{ij}
    (\mathcal K+\theta)
  -\sigma_{ij} \theta
\Big].
\end{align*}
We rewrite the unit normal vector field
\[
   -\frac{\omega}{\rho} \nu
   = \frac{1}{\rho}e^{\gamma \,\mathsf A}(\alpha\theta -D \gamma).
\]
We use the following definition of second fundamental form
\[
h_{ij} =- \langle \overline \nabla _{e_j}\overline \nabla _{e_i}X, \nu\rangle
\]
and let
\[
h_{ij} :=\frac{\rho}{\omega}\mathfrak h_{ij}.
\]

Since \(e^{\gamma \mathsf A}\) is an isometry and
\begin{align*}
    &\langle \mathcal K,\theta\rangle =\langle\mathsf A \theta,\theta\rangle= 0, \quad
    \langle \theta,\theta\rangle =1 , \quad
 \langle e_i,\theta\rangle =0,\quad \\
    &\langle\mathsf A \mathcal K,\theta\rangle=-\langle \mathcal K,\mathsf A\theta\rangle= -|\mathcal K|^2_{\sigma}, \quad
\langle\mathsf A e_i,\theta\rangle=-\langle e_i,\mathsf A \theta\rangle =-\mathcal K_i, \quad
    \\
    &\langle \mathcal K,D \gamma \rangle =\langle \mathsf A \theta,D \gamma \rangle =\mathcal K \gamma ,\quad
      \langle e_i,D \gamma \rangle =\gamma_i ,\quad
\langle \theta,D \gamma \rangle =0,\quad
\quad\\
    &\langle\mathsf A \mathcal K,D \gamma   \rangle
    =
    -\langle\mathcal K,\mathsf A D \gamma   \rangle , \quad
\langle\mathsf A e_i,D \gamma \rangle= -\langle e_i,\mathsf A D \gamma \rangle
= -(\mathsf A D \gamma )_i, \quad
    \\
\end{align*}
 we obtain
\begin{align*}
    -\mathfrak h_{ij}
=\langle \Big[
 & \gamma _i \gamma_j(2\mathcal K+\theta)
+ \gamma _i \gamma_j \mathsf A\mathcal K
 \\
 &+\gamma _j(e_i+\mathsf A e_i)
+\gamma _i(e_j+\mathsf A e_j)
 \\
 &
     + \gamma_{ij}
    (\mathcal K+\theta)
    -\sigma_{ij} \theta
 \Big], \alpha\theta -D \gamma\rangle \\
    =&
    \alpha
    \Big[
    \gamma _i \gamma _j(1-|\mathcal K|^2_{\sigma})
    -\gamma _j \mathcal K_i -\gamma _i \mathcal K_j
    + \gamma _{ij}
    - \sigma_{ij}
\Big]
\\
     &
     +\gamma _i \gamma _j(-2\mathcal K \gamma+\langle \mathcal K, \mathsf A D \gamma \rangle )\\
     &-2\gamma _i \gamma _j
     +\gamma _j(\mathsf A D\gamma)_i
     +\gamma _i(\mathsf A D\gamma)_j
     \\
     &-\gamma _{ij}\mathcal K \gamma.\\
\end{align*}
Thus,
\begin{align*}
    -\mathfrak h_{ij}
        =&\gamma _{ij} -\alpha \sigma_{ij}\\
     & +\gamma _i \gamma _j(-\alpha(1+|\mathcal K|^2_{\sigma})
     +\langle \mathcal K, \mathsf A D \gamma \rangle )\\
     &-\alpha(\gamma _i\mathcal K_j+\gamma _j\mathcal K_i)
     +\gamma _j(\mathsf A D\gamma)_i
     +\gamma _i(\mathsf A D\gamma)_j.
\end{align*}
We can calculate the second fundamental form and the Weingarten tensor under
the spiral radial parametrization.
\begin{lemma}
    \label{lem:h-and-W} The second fundamental form and Weingarten tensor satisfy
\begin{align*}
    h_{ij} =& \frac{\rho}{\omega}\mathfrak h_{ij}
    =\frac{\rho}{\omega}
    \Big\{
        -\gamma _{ij}
             +\gamma _i \gamma _j
        \Big[
            \alpha(1+|\mathcal K|^2_{\sigma})
            -\langle \mathcal K, \mathsf A D \gamma \rangle )
        \Big]\\
     &\qquad+\alpha(\gamma _i\mathcal K_j+\gamma _j\mathcal K_i)
     -\gamma _j(\mathsf A D\gamma)_i
     -\gamma _i(\mathsf A D\gamma)_j\\
            &\qquad+\alpha \sigma_{ij}
    \Big\}
    \\
    h_{i}{}^j =&    \frac{1}{\rho\omega}
\Big(
\sigma^{jl} +\mathcal K^j\mathcal K^l-
\frac{(\gamma ^j+\alpha \mathcal K^j )(\gamma ^l+\alpha \mathcal K^l )}{\omega^2}\Big)
               \cdot \mathfrak h_{il}
\end{align*}
where $\omega := \sqrt{\alpha^2+|D \gamma |^2}$, \(\alpha:=1+\mathcal K \gamma \)
and \( \mathcal K \gamma= \langle \mathcal K, D \gamma \rangle_{\sigma}  \).
\end{lemma}

\section{Initial Value Problem and \texorpdfstring{\(  C^0 \)}{C0} estimate}
\label{sec:ivp-C0}

Similar as in~\cite{GuanLi2026}, we derive the initial value problem for \( \gamma  \)
under a time-independent spiral parametrization
\( X= \rho e^{\gamma \mathsf A}\theta\)
where \( \theta \) is independent of \( t \).

Along flow
\[
\partial _t X=f\nu,
\]
direct calculations using \( \gamma _t = \partial _t \gamma (\theta, t) \) yield
\begin{align*}
    \partial _t X = \gamma _t (X+\mathcal K).
\end{align*}
Thus,
\begin{align*}
   \langle \partial _t X,\nu \rangle
     = \gamma _t \widehat u \quad
     \mathrm{and}\quad
\langle \partial _t X,\nu \rangle = f.
\end{align*}
We have
\[
\gamma _t = \frac{f}{\widehat u }.
\]

Therefore, the flow~\eqref{eflow} is equivalent to the initial-value problem
\begin{equation}
\begin{aligned}
  \partial_t\gamma
  &=\frac{1}{\widehat u  F}-\frac{1}{c_{n,k}\alpha},
\end{aligned}
\label{eq:gamma-IVP}
\end{equation}
where we have used~\eqref{eq:U-and-u}.

We next show that the radial function $\rho$ is uniformly bounded from
above and below.

\begin{proposition}
\label{prop:C0-rho}
Let $M_0$ be a radial graph with radial function $\rho_0$ over
$\mathbb S^n$ in $\mathbb R^{n+1}$. If $\gamma (\theta,t)$ solves
\eqref{eq:gamma-IVP}, then
\begin{equation}
  \min_{\mathbb S^n}\rho_0
  \le \rho(z,t)
  \le \max_{\mathbb S^n}\rho_0
  \qquad
  \text{for every }(z,t)\in\mathbb S^n\times[0,T).
  \label{eq:C0-rho-bound}
\end{equation}
\end{proposition}

\begin{proof}
Since $\dd\gamma/\dd\rho=1/\rho>0$, the functions $\gamma$ and $\rho$
have the same spatial maximum and minimum points. At any spatial
critical point of $\gamma$,
\[
  \nabla\gamma=0,
  \qquad
  \alpha=1,
  \qquad
  \omega=1,
  \qquad
  \widehat u  = \rho.
\]
It follows from~\Cref{lem:h-and-W} that
\[
  h_i{}^j
  =\frac{1}{\rho}
  (
  1 -\gamma _{i}{}^j
  )
  .
\]

At a spatial maximum of $\gamma$, we have
$\gamma_i{}^j\le0$, and hence
\[
  h_i{}^j\ge \frac{1}{\rho}\delta_i{}^j.
\]
By the ellipticity, monotonicity, and homogeneity of $F$,
\[
  F(h_i{}^j)
  \ge \frac{1}{\rho}F(I)
  =\frac{c_{n,k}}{\rho}.
\]
Equation \eqref{eq:gamma-IVP} therefore gives
\[
  \partial_t\gamma
=\frac{1}{\widehat u  F}-\frac{1}{c_{n,k}\alpha}
=\frac{1}{\rho F}-\frac{1}{c_{n,k}}
  \le0
\]
at a spatial maximum.

Similarly, at a spatial minimum of $\gamma$, we have
\[
  \partial_t\gamma\ge0.
\]
The parabolic maximum principle now gives
\[
  \min_{\mathbb S^n}\gamma(\cdot,0)
  \le\gamma(z,t)
  \le\max_{\mathbb S^n}\gamma(\cdot,0).
\]
Since $\gamma$ is strictly increasing as a function of $\rho$, this is
equivalent to \eqref{eq:C0-rho-bound}.
\end{proof}

\section{Gradient estimates}
\label{sec:gradient-estimate}

\subsection{Identities for \texorpdfstring{\( \widehat u  \)}{rotational u } }
\label{subsec:identities-U}

Recall \( \widehat u  =u+\mathring u\).
We first derive the Hessian of \(\mathring u\),
rotational part of the support function.

Under an orthonormal frame \( \{e_i\} \) with respect to the induced metric
\( g_{ij} \) of the hypersurface \( M_0 \) at a given point.
We calculate
\begin{align*}
    \nabla _i \mathring u =&
    \nabla _i \langle \mathcal K, \nu\rangle
    =\nabla _i \langle \mathsf A X, \nu\rangle\\
    =&
    \langle \mathsf A X_i, \nu \rangle
    +
    h_i{}^l\langle \mathsf A X, X _l \rangle.
\end{align*}
Take one more derivative, we obtain the Hessian equation.
\begin{align*}
    \nabla _j\nabla _i \mathring u =&
    \langle \mathsf A \overline\nabla _j X_i, \nu \rangle
    +h_j{}^l\langle \mathsf A X_i, X_l \rangle\\
                                    &
+ \nabla _jh_i{}^l\langle \mathsf A X, X _l \rangle
+h_i{}^l\langle \mathsf A X_j, X _l \rangle
+h_i{}^l\langle \mathsf A X,\overline\nabla _j X _l \rangle
    .
\end{align*}

We first note that
\[
    \langle \mathsf A \overline\nabla _j X_i, \nu \rangle=
\langle \mathsf A \nabla^{g} _j X_i, \nu \rangle
+h_{ij}\langle \mathsf A \nu, \nu \rangle
    =0.
\]
while \( \nabla^{g} _j X_i=0 \) because of the orthonormal frame and \(\langle\mathsf A\nu,\nu \rangle =0  \) because the skew-symmetry of \( \mathsf A \). Similarly, we have
\[
h_i{}^l\langle \mathsf A X,\overline\nabla _j X _l \rangle
=-h_i{}^lh_{jl}\langle \mathsf A X,\nu\rangle
=-h_i{}^lh_{jl}\mathring u.
\]

Thus, we obtain
\begin{align*}
    \nabla _j\nabla _i \mathring u =&
 \nabla _jh_i{}^l\langle \mathcal K, X _l \rangle
-h_i{}^lh_{jl}\mathring u\\
                                    &+h_i{}^l\langle \mathsf A X_j, X _l \rangle
+h_j{}^l\langle \mathsf A X_i, X_l \rangle.
\end{align*}
Hessian equation of the radial part of the suport function \( u \) is well-known,
\[
\nabla _j\nabla _i u =
 \nabla _jh_i{}^l\langle X, X _l \rangle
-h_i{}^lh_{jl}u
+h_{ij}.
\]

Put these two parts together, we obtain
\begin{align*}
    \nabla _j\nabla _i \widehat u  =&
 \nabla _jh_i{}^l\langle Y, X _l \rangle
-(h^2)_{ij}\widehat u
+h_{ij}\\
                                    &+h_i{}^l\langle \mathsf A X_j, X _l \rangle
+h_j{}^l\langle \mathsf A X_i, X_l \rangle
    .
\end{align*}
Using the Codazzi property, we get
\begin{align*}
    \nabla _j\nabla _i \widehat u  =&
 \langle  \nabla h_{ij},Y^\top\rangle
-(h^2)_{ij}\widehat u
+h_{ij}\\
                                    &+h_i{}^l\langle \mathsf A X_j, X _l \rangle
+h_j{}^l\langle \mathsf A X_i, X_l \rangle
    .
\end{align*}

Moreover, using the skew-symmetry of \( \mathsf A \),
after we contract the Hessian with \( g^{ij} \) or \( F^{ij} \),
the last two terms in the Hessian will vanish.
Consequently, we have the following identities.

\begin{lemma}
    \label{lem:spiral-u-identities}
The following identities hold.
\begin{equation}
\begin{aligned}
    F^{ij} \nabla _j \nabla _i \widehat u  =
 &\langle  \nabla F,Y^\top\rangle
-\dot F(h^2)\widehat u
+F\\
   \Delta  \widehat u  =
 &\langle  \nabla H,Y^\top\rangle
-|A|^2\widehat u
+H.
\end{aligned}
\end{equation}
\end{lemma}

\subsection{Gradient estimates }
\label{subsec:evolution-U}

By direct calculation, we have the following lemma.
\begin{lemma}
    \label{lem:evolution-U-general}
    Along the following flow,
    \begin{equation}
        \label{general-flow}
        \partial _t X = f\nu,
    \end{equation}
we have
    \begin{equation}
        \label{eq:evolution-U-general}
        \partial _t \widehat u  =
        f- \langle Y^{\top}, \nabla f \rangle.
    \end{equation}
\end{lemma}

The following simple identity plays a crucial role in the evolution of \( \widehat u  \).
\begin{lemma}
    \label{lem:spiral-support}
    The rotational part support function satisfies
    \begin{equation}
        \label{eq:spiral-u-identity-1}
        \mathring u = \langle \mathcal K,\nu\rangle
        =X^{\top}(\widehat u ) -Y^{\top}(u).
    \end{equation}
    Equivalently,
    \begin{equation}
        \label{eq:spiral-u-identity-2}
        Y^{\top}(u)-u =X^{\top}(\widehat u ) -\widehat u.
    \end{equation}
\end{lemma}
\begin{proof}
   For simplicity, we let
   \begin{align*}
       X^{\top}=x_i X_i,\qquad Y^{\top}=y_iX_i.
   \end{align*}
We also have \( \nabla \widehat u  =\nabla _j \widehat u \, X_j \)
and \( \nabla u =\nabla _j u \, X_j  \).
   Thus
   \begin{align*}
       X^{\top}(\widehat u ) = x_i \nabla _i\widehat u,\quad
       Y^{\top}(u) = y_i \nabla _iu.
   \end{align*}
We calculate
   \begin{align*}
       \nabla_i \widehat u  =&
       \langle \mathsf A X_i, \nu\rangle + \langle Y,\nabla _i \nu\rangle
       = \langle \mathsf A X_i, \nu\rangle +y_j \langle X_j,\nabla _i \nu\rangle\\
       \nabla_i u =&
        \langle X,\nabla _i \nu\rangle
       = x_j \langle X_j,\nabla _i \nu\rangle.
   \end{align*}
   Consequently, we have
   \begin{align*}
       X^{\top}(\widehat u ) -Y^{\top}(u)=&
x_i\langle \mathsf A X_i, \nu\rangle
+x_iy_j h_{ji}
-y_ix_j  h_{ji}\\
       =&\langle \mathsf A X, \nu\rangle,
   \end{align*}
  where we have used that the second fundamental form is symmetric in the last step.
  This complets the proof since \( \mathcal K = \mathsf A X \).
\end{proof}

We let
\begin{equation}
    \label{eq:L}
    \mathcal L:=
        \partial _t
    - \frac{1}{F^2}F^{ij} \nabla _j \nabla _i
    -\frac{1}{c_{n,k}}\langle X^{\top},\nabla \rangle.
\end{equation}
We obtain the following lemma which implies the uniform $C^1$ bound.
\begin{lemma}
    \label{lem:evolution-U}
    Along the inverse curvature flow~\eqref{eflow}, we have
    \begin{equation}
        \label{eq:evolution-U}
        \mathcal L \widehat u
    =
\left(
    \frac{\dot F(h^2)}{F^2} - \frac{1}{c_{n,k}}
    \right)
\widehat u .
    \end{equation}
    Moreover, let $M_t\subset\mathbb R^{n+1}$ be a smooth admissible solution
    of \eqref{eflow} for \( t\in[0,T) \).
Then
\[
  \widehat u_t \ge \min_{M_0} \widehat u_0.
\]
\end{lemma}
\begin{proof}
   We have
   \[
   \partial _t \widehat u  =
   \frac{1}{F} -\frac{u}{c_{n,k}}
   -\langle Y^{\top}, \nabla
   \left(
\frac{1}{F} -\frac{u}{c_{n,k}}
       \right)
   \rangle
   \]
  and
  \[
     F^{ij} \nabla _j \nabla _i \widehat u  =
 \langle  \nabla F,Y^\top\rangle
-\dot F(h^2)\widehat u
+F.
  \]
 This implies
 \begin{align*}
    \partial _t \widehat u
    - \frac{1}{F^2}F^{ij} \nabla _j \nabla _i \widehat u  =&
\frac{\dot F(h^2)}{F^2}\widehat u
+\frac{1}{c_{n,k}}
\left(
    \langle Y^{\top},\nabla u \rangle -u.
    \right)
 \end{align*}
 By~\Cref{lem:spiral-support},
 \[
 \langle Y^{\top},\nabla u \rangle -u =X^{\top}(\widehat u )-\widehat u,
 \]
 we finish the proof of the evolution equation~\eqref{eq:evolution-U}.
 The uniform lower bound is by standard maximum principle and
\[
\frac{\dot F(h^2)}{F^2} \ge \frac{1}{c_{n,k}},
\]
which can be found in Lemma~4.1 from \cite{GuanLi2026}.
\end{proof}

With the \( C^0 \) estimate from~\Cref{prop:C0-rho},
uniform lower bound of \( \widehat u  \), and~\eqref{eq:U-and-u},
we obtain
\begin{proposition}
\label{prop:C1-rho}
Let $M_0$ be a radial graph with radial function $\rho_0$ over
$\mathbb S^n$ in $\mathbb R^{n+1}$. If $\gamma (\theta,t)$ solves
\eqref{eq:gamma-IVP}, then
\begin{equation}
  \sup_{(\theta,t)\in\mathbb S^n\times [0,T)}|\nabla \rho|(\theta,t)
  \le \max_{\theta\in \mathbb S^n} |\nabla \rho|(\theta,0).
  \label{eq:C1-rho-bound}
\end{equation}
\end{proposition}

\subsection{Lower bound for \texorpdfstring{\( F \)}{F}}
\label{subsec:F}

To obtain a lower bound for \( F \),
we follow the similar techniques as in~\cite{GuanLi2026}.
For simplicity, we let
\begin{align*}
    b_{i}{}^j =&
\Big(
\omega^2(\sigma^{jl} +\mathcal K^j\mathcal K^l)-
{(\gamma ^j+\alpha \mathcal K^j )(\gamma ^l+\alpha \mathcal K^l )}
\Big)\\
               &\quad
              \cdot \Big(
                   -\gamma _{ij}
             +\gamma _i \gamma _j
        \big[
            \alpha(1+|\mathcal K|^2_{\sigma})
            -\langle \mathcal K, \mathsf A D \gamma \rangle )
        \big]\\
     &\qquad+\alpha(\gamma _i\mathcal K_j+\gamma _j\mathcal K_i)
     -\gamma _j(\mathsf A D\gamma)_i
     -\gamma _i(\mathsf A D\gamma)_j\\
            &\qquad+\alpha \sigma_{ij}
    \Big),
\end{align*}
and
\begin{align*}
      \widetilde F
  :=F(b_i{}^j)
  =\frac{\sigma_{k+1}(b_i{}^j)}
         {\sigma_k(b_i{}^j)}.
\end{align*}
Then,
\begin{align*}
    h_i{}^j = \frac{1}{\rho \omega^3} b_i{}^j
    \quad \mathrm{and}\quad
    \widetilde F =\rho \omega^3 F.
\end{align*}
\begin{proposition}
\label{prop:F-lower-bound}
Let $M_t\subset\mathbb R^{n+1}$ be a smooth admissible solution of \eqref{eflow}.
Then there exists a constant $C_0>0$, depending only on the initial
hypersurface and the gradient bound, such that
\[
  F\ge C_0
\]
on $\mathbb S^n\times[0,T)$.
\end{proposition}

\begin{proof}
Set
\[
  \dot\gamma:=\partial_t\gamma.
\]
By \eqref{eq:gamma-IVP},
\begin{equation}
\gamma_t
=  \partial_t\gamma
=\frac{\omega^4}{\widetilde F}-\frac{1}{c_{n,k}\alpha}.
  \label{eq:dotgamma}
\end{equation}
We perform the following computation at a spatial critical point $P$
of $\dot\gamma$. At $P$,
\[
  \nabla\dot\gamma=0,
  \qquad
  \partial _t \alpha = \langle \mathcal K, D \dot \gamma \rangle=0,
  \qquad
  \partial_t\omega^2
  = 2 \alpha \alpha_t
  +2\inner{\nabla\gamma}{\nabla\dot\gamma}=0.
\]
Differentiating \eqref{eq:dotgamma} therefore gives
\begin{equation}
  \partial_t\dot\gamma
  =
   -\frac{\omega^4}{\widetilde F^2}
    \partial_t\widetilde F.
  \label{eq:dotgamma-time}
\end{equation}
At $P$, \begin{equation}
\begin{aligned}
  \partial_t\widetilde F
  &=-\widetilde F^{ij}
\Big(
\omega^2(\sigma^{jl} +\mathcal K^j\mathcal K^l)-
{(\gamma ^j+\alpha \mathcal K^j )(\gamma ^l+\alpha \mathcal K^l )}
\Big)
    (\dot\gamma)_{i\ell} .
\end{aligned}
\label{eq:Ftilde-time}
\end{equation}
Combining \eqref{eq:dotgamma-time} and
\eqref{eq:Ftilde-time}, we obtain at $P$
\begin{align}
  &\partial_t\dot\gamma
  -\frac{\omega^4}{\widetilde F^2}
   \widetilde F^{ij}
\Big(
\omega^2(\sigma^{jl} +\mathcal K^j\mathcal K^l)-
{(\gamma ^j+\alpha \mathcal K^j )(\gamma ^l+\alpha \mathcal K^l )}
\Big)
(\dot\gamma)_{i\ell}=0.
  \label{eq:dotgamma-parabolic}
\end{align}
Maximum principle implies that
\[
\frac{\omega^4}{\widetilde F }
=\gamma _t+\frac{1}{c_{n,k}\alpha}
\le C,
\]
hence
\[
\widetilde F\ge \frac{\omega^4}{C}\ge c\quad \mathrm{and}\quad F=\frac{\widetilde F}{\rho \omega^3}\ge c'
\]
which completes the proof.
\end{proof}

\section{Curvature estimates}
\label{sec:curvature-estimate}

In this section we derive the curvature estimates required for the
long-time existence of the flow. Since the geometric evolution equation
is the same as the one considered in~\cite{GuanLi2026, GuanLi2021}, the evolution equations
for the geometric quantities can be taken directly from there. The new
point here is that the $C^0$ and $C^1$ estimates obtained in the previous
sections allow us to apply these estimates to the spiral-star-shaped setting.

Recall that
\[
    F=\frac{\sigma_{k+1}}{\sigma_k},
    \qquad
    c_{n,k}=\frac{n-k}{k+1}.
\]

\begin{proposition}
\label{prop:evolution-h}
\cite{GuanLi2026}
Along the flow~\eqref{eflow} in $\mathbb R^{n+1}$, the Weingarten
tensor satisfies
\begin{align}
\mathcal L h_i{}^j
  &=
  \frac{1}{F^2}\ddot F
  \big(
    \nabla_i h,\nabla^j h
  \big)
  -\frac{2}{F^3}\nabla_iF\nabla^jF
  -\frac{2}{F}(h^2)_i{}^j
  \notag\\
  &\quad
  +\left[
     \frac{\dot F(h^2)}{F^2}
     +\frac{1}{c_{n,k}}
   \right]h_i{}^j ,
  \label{eq:evolution-h}
\end{align}
where $\mathcal L$ is defined in~\eqref{eq:L}.
\end{proposition}

We first obtain a uniform upper bound for $F$.

\begin{proposition}
\label{prop:F-upper-bound}
Let $M_t\subset\mathbb R^{n+1}$ be a smooth admissible solution of
\eqref{eflow}. Then
\[
    F(\cdot,t)\le \max_{M_0}F
\]
for every $t\in[0,T)$.
\end{proposition}

\begin{proof}
Contracting~\eqref{eq:evolution-h} with $F^{ij}$ yields
\begin{align}
\mathcal L F
  &=
  \frac{1}{F^2}
  F^{ij}\,
  \ddot F
  \big(
     \nabla_i h,\nabla^j h
  \big)
  -\frac{2}{F^3}
  F^{ij}\nabla_iF\nabla_jF
  \notag\\
  &\quad
  -\left[
     \frac{\dot F(h^2)}{F^2}
     -\frac{1}{c_{n,k}}
   \right]F .
  \label{eq:evolution-F}
\end{align}
Since $F=\sigma_{k+1}/\sigma_k$ is concave in the admissible cone,
the first term on the right-hand side is nonpositive. Moreover,
\[
   F^{ij}\nabla_iF\nabla_jF\ge0,
\]
and the algebraic inequality
\[
   \dot F(h^2)\ge \frac{1}{c_{n,k}}F^2
\]
implies that the last term in~\eqref{eq:evolution-F} is also
nonpositive. Hence
\[
   \mathcal L F\le0
\]
at a spatial maximum of $F$. The maximum principle gives
\[
   \max_{M_t}F\le \max_{M_0}F,
\]
as claimed.
\end{proof}

Together with Proposition~\ref{prop:F-lower-bound}, we therefore have
\begin{equation}
    0<C_1\le F\le C_2<\infty
    \qquad\text{on }\mathbb S^n\times[0,T),
    \label{eq:F-two-sided}
\end{equation}
where the constants depend only on the initial hypersurface.

We next estimate the principal curvatures. Recall that
$\widehat u$ denotes the generalized support function introduced in
\eqref{eq:U-and-u}. By the $C^0$ and $C^1$ estimates established
previously,
\begin{equation}
    0<u_0\le \widehat u\le u_1<\infty ,
    \label{eq:uhat-two-sided}
\end{equation}
where \( u_0 \) and \( u_1 \) are uniform constants only depending on the
initial hypersurface.

\begin{proposition}
\label{prop:evolution-h-over-u}
Along the flow~\eqref{eflow},
\begin{align}
\mathcal L\left(
    \frac{h_i{}^j}{\widehat u}
    \right)
  &=
  \frac{1}{\widehat u F^2}
  \ddot F
  \big(
    \nabla_i h,\nabla^j h
  \big)
  +\frac{2}{\widehat u F^2}
    F^{kl}
    \nabla_k\left(
       \frac{h_i{}^j}{\widehat u}
    \right)
    \nabla_l\widehat u
  \notag\\
  &\quad
  -\frac{2}{\widehat u F^3}
   \nabla_iF\nabla^jF
  -\frac{2}{\widehat u F}(h^2)_i{}^j
  +\frac{2}{c_{n,k}}
   \frac{h_i{}^j}{\widehat u}.
  \label{eq:evolution-h-over-u}
\end{align}
\end{proposition}

\begin{proof}
Using the quotient rule for $\mathcal L$, we obtain
\[
\begin{aligned}
\mathcal L\left(
    \frac{h_i{}^j}{\widehat u}
\right)
&=
\frac{1}{\widehat u}\mathcal L h_i{}^j
-\frac{h_i{}^j}{\widehat u^2}\mathcal L\widehat u  \\
&\quad
+\frac{2}{\widehat u F^2}
F^{kl}
\nabla_k\left(
   \frac{h_i{}^j}{\widehat u}
\right)
\nabla_l\widehat u .
\end{aligned}
\]
Substituting the evolution equations
\eqref{eq:evolution-U} and~\eqref{eq:evolution-h} gives
\eqref{eq:evolution-h-over-u}.
\end{proof}

\begin{corollary}[Uniform curvature estimate]
\label{cor:C2-bound}
Let $M_t$ be a smooth admissible solution of~\eqref{eflow} for
$t\in[0,T)$. Then there exists a constant $C>0$, depending only on
the initial hypersurface, such that
\begin{equation}
    |h|\le C
    \qquad\text{on }M_t,
    \qquad t\in[0,T).
    \label{eq:curvature-bound}
\end{equation}
In particular,
\begin{equation}
    \|\rho(\cdot,t)\|_{C^2(\mathbb S^n)}
    \le C,
    \qquad t\in[0,T).
    \label{eq:spherical-C2-bound}
\end{equation}
\end{corollary}

\begin{proof}
Set
\[
    W_i{}^j:=\frac{h_i{}^j}{\widehat u}.
\]
Let $\Lambda$ denote the largest eigenvalue of $W_i{}^j$, and suppose
that $\Lambda$ attains a spatial maximum at $(x_0,t_0)$. Choose a local
orthonormal frame at $(x_0,t_0)$ in which $h_i{}^j$ is diagonal and
\[
    \Lambda=W_1{}^1
    =\frac{\kappa_{\max}}{\widehat u}.
\]
By the tensor maximum principle, at this point the first-order drift
term in~\eqref{eq:evolution-h-over-u} does not contribute to the
evolution of $\Lambda$. Since $F$ is concave,
\[
    \ddot F(\nabla_1h,\nabla^1h)\le0,
\]
and the gradient term involving $\nabla_1F$ is also nonpositive.
Consequently,
\begin{align*}
  \mathcal L\Lambda
  &\le
  -\frac{2}{\widehat u F}\kappa_{\max}^2
  +\frac{2}{c_{n,k}}\Lambda \\
  &=
  -\frac{2\widehat u}{F}\Lambda^2
  +\frac{2}{c_{n,k}}\Lambda .
\end{align*}
Using~\eqref{eq:F-two-sided} and~\eqref{eq:uhat-two-sided}, we obtain
\[
  \mathcal L\Lambda
  \le
  -\frac{2u_0}{C_2}\Lambda^2
  +\frac{2}{c_{n,k}}\Lambda .
\]
Hence the maximum principle implies
\[
  \Lambda
  \le
  \max\left\{
      \max_{M_0}\frac{\kappa_{\max}}{\widehat u},
      \frac{C_2}{c_{n,k}u_0}
  \right\}.
\]
Together with the upper bound for $\widehat u$, this gives
\[
    \kappa_{\max}\le C.
\]

Since the solution remains admissible,
\[
    \kappa\in\Gamma_{k+1}\subset\Gamma_1,
\]
and therefore
\[
    H=\sum_{i=1}^n\kappa_i>0.
\]
It follows that
\[
    \kappa_{\min}
    \ge -(n-1)\kappa_{\max},
\]
and hence
\[
    |h|\le C.
\]

Finally, the expression of the Weingarten tensor in terms of the
radial graph $\rho$, together with the uniform $C^0$ and $C^1$
estimates, shows that the curvature bound is equivalent to a uniform
bound for $D^2\rho$. This proves~\eqref{eq:spherical-C2-bound}.
\end{proof}

\section{Convergence}

The convergence argument is entirely geometric and is unchanged from
that in~\cite{GuanLi2026, GuanLi2021,BGL}. We recall the conclusion.

\begin{theorem}[Exponential convergence]
\label{thm:convergence}
Under the assumptions of \Cref{thm:main}, the
hypersurfaces $M_t$ converge smoothly and exponentially fast to a
round sphere as $t\to\infty$.

More precisely, there exists a constant $\rho_\infty>0$ such that,
after expressing $M_t$ in the spiral radial coordinates used
above,
\[
   \rho(\cdot,t)\longrightarrow \rho_\infty
   \qquad\text{in }C^\infty(\mathbb S^n),
\]
and for every $m\ge0$ there exist constants $C_m,\lambda>0$ such that
\begin{equation}
   \|\rho(\cdot,t)-\rho_\infty\|_{C^m(\mathbb S^n)}
   \le C_m e^{-\lambda t}.
   \label{eq:exponential-convergence}
\end{equation}
Equivalently, $M_t$ converges exponentially in $C^\infty$ to the
corresponding round sphere.
\end{theorem}

\begin{proof}[Proof of \Cref{thm:main}]
By Propositions~\ref{prop:C1-rho},
\ref{prop:F-lower-bound}, \ref{prop:F-upper-bound}, and
Corollary~\ref{cor:C2-bound}, the radial function satisfies uniform
$C^2$ estimates and
\[
    0<C_1\le F\le C_2.
\]
Thus the scalar equation~\eqref{eq:gamma-IVP} remains uniformly
parabolic. Standard parabolic regularity gives uniform estimates of
all higher derivatives, and the solution therefore exists for all
$t\ge0$.

At this point the geometric flow is precisely the flow studied
in~\cite{GuanLi2009}. The long-time convergence argument there applies
without change. In particular, the gradient decays exponentially, the
oscillation of the radial function tends to zero, and the preserved
quermassintegral determines a unique limiting radius $\rho_\infty$.
Consequently,
\[
    M_t\longrightarrow S_{\rho_\infty}
\]
smoothly and exponentially fast as $t\to\infty$.
\end{proof}

\section*{Acknowledgments}
The author thanks Pengfei Guan for bringing the work of Flynn and Reznikov
to his attention and for numerous helpful discussions.
He also thanks Jacob Reznikov for valuable explanations of that work.
The author gratefully acknowledges the hospitality of the Centre de
recherches math\'ematiques (CRM) and McGill University, as well as
the support of CRM and the Simons Foundation
 during his visit as a CRM-Simons professor in 2024.

\end{document}